\documentclass[12pt, reqno]{amsart}
\usepackage{amsmath, amsthm, amscd, amsfonts, amssymb, graphicx, color}
\usepackage[bookmarksnumbered, colorlinks, plainpages]{hyperref}
\input{mathrsfs.sty}
\usepackage[T1]{fontenc}
\usepackage[polish, english]{babel}
\usepackage[utf8]{inputenc}

\hypersetup{colorlinks=true,linkcolor=red, anchorcolor=green,
citecolor=cyan, urlcolor=red, filecolor=magenta, pdftoolbar=true}

\newtheorem{twr}{Theorem}[section]
\newtheorem{lem}[twr]{Lemma}

\theoremstyle{definition}

\newtheorem{question}[twr]{Question}
\theoremstyle{remark}

\numberwithin{equation}{section}

\DeclareMathOperator{\id}{Id}

\DeclareMathOperator{\fix}{Fix}

\usepackage{color}
\usepackage{enumerate}

\begin{document}

\setcounter{page}{1}

\title[Normed spaces without small-norm projections]{Finite-dimensional normed spaces without small-norm hyperplane projections}
\author[T. Kobos]
{Tomasz Kobos$^1$}
\address{$^1$Faculty of Mathematics and Computer Science, Jagiellonian University, \L ojasiewicza 6, 30-348 Krak\'ow, Poland}
\email{tomasz.kobos@uj.edu.pl}

%\subjclass[2010]{47A58, 41A65, 52A21.}
%\keywords{Minimal projection; Affine dimension; Polyhedral norm; Chalmers-Metcalf operator}
%%%%%%%%%%%%%%%%%%%%%%%%%%%%%%%%%%%%%%%%%%%%%%
\begin{abstract} We prove that there exists an absolute constant $c>1$ such that for every $n \geq 3$ there exists an $n$-dimensional real normed space $X$ for which every linear projection $P$ onto a hyperplane satisfies $\|P\| \geq c$. More generally, for every integer $k \geq 1$ there exists a constant $c_k>1$ such that for every $n \geq k+2$ there exists an $n$-dimensional real normed space $X$ such that every projection $P$ onto a subspace of codimension at most $k$ satisfies $\|P\| \geq c_k$. The proof is based on an idea of Gleit and Lazar, who established an analogous result for the space of real $n \times n$ matrices equipped with either the trace-class norm or the operator norm. \end{abstract}

\maketitle
%%%%%%%%%%%%%%%%%%%%%%%%%%%

\section{Introduction}

%only real case???

Let $X$ be a real normed space and let $Y\subseteq X$ be a linear subspace. A bounded linear operator $P:X\to Y$ is called a projection onto $Y$ if $P|_Y=\id_Y$. The quantity 
$$\lambda(Y,X)=\inf\{\|P\|:\ P:X\to Y \text{ is a projection}\}$$ 
is called the \emph{relative projection constant} of $Y$ in $X$. Norm-one projections play a distinguished role in the geometry of Banach spaces. They may be regarded as a natural analogue of orthogonal projections in Euclidean spaces and for this reason have been investigated extensively. Subspaces satisfying $\lambda(Y,X)=1$ are called \emph{$1$-complemented}. The theory of $1$-complemented subspaces has been developed over several decades and constitutes a central part of the local theory of Banach spaces. A comprehensive overview of this topic may be found in the survey of Randrianantoanina \cite{1complementedsurvey}. 

A natural question is how common norm-one projections are among finite-dimensional normed spaces. While every $1$-dimensional subspace is automatically $1$-complemented by Hahn--Banach theorem, it turns out that norm-one projections onto subspaces of larger dimensions are highly exceptional. The most classical finite-dimensional normed spaces, such as the spaces $\ell_p^n$, typically possess some $1$-complemented subspaces and, in many cases, these subspaces are completely characterized. Nevertheless, already in 1941 Bohnenblust \cite{bohnenblust} proved that a sufficiently generic subspace of $\ell_p^n$ does not contain any non-trivial $1$-complemented subspaces. An even more fundamental result was obtained by Bosznay and Garay \cite{bosznaygaray} in 1986. They proved that for every $n\geq 3$ the set of $n$-dimensional real normed spaces without non-trivial $1$-complemented subspaces is open and dense in the $n$-dimensional Banach--Mazur compactum. Thus, from the generic point of view, norm-one projections are exceptional objects. Bosznay and Garay also raised the problem of obtaining precise, quantitative estimates for projection constants, asking how far the norms of projections can be bounded away from $1$.

Obtaining such quantitative estimates appears to be considerably more difficult than merely proving the nonexistence of norm-one projections. It seems that the first result in this direction was obtained in \cite{koboslp}, where a quantitative version of Bohnenblust's theorem was established for hyperplane projections in suitable subspaces of $\ell_p^n$. As a consequence, for every $n \geq 4$ there exists an $n$-dimensional real normed space $X$ such that
\begin{equation}
\label{estimate}
\lambda(Y, X) > 1 + \left (8 \left ( n + 3 \right )^{5} \right )^{-30(n+3)^2} > 1 + \exp(-C n^2 \log n),
\end{equation}
for every hyperplane $Y \subseteq X$.

Although this estimate is quantitative, the lower bound approaches $1$ very rapidly as the dimension increases. A lower bound of a comparable order was obtained in \cite{kobospolytope}, where the ambient space was assumed to have a polyhedral unit ball with many vertices lying on the unit sphere of $\mathbb R^n$. In that setting the norm of every hyperplane projection can be estimated explicitly in terms of suitable determinants associated with the vertices of the polytope and with the vectors corresponding to its facets. The minimum of these determinants may be viewed as a measure of the polytope being sufficiently generic, which in turn allows one to estimate the probability that a randomly chosen polytope yields a prescribed lower bound on the norms of hyperplane projections. However, since the construction requires an exponential number of vertices, the resulting lower bound is again necessarily very close to $1$.

More recently, the problem of obtaining uniform lower bounds for projection constants appeared in a different context. Epperlein and Wirth in \cite{epperleinwirth} related projection constants to the theory of the joint spectral radius of matrix families. They introduced the notion of a \emph{shady} norm and showed that large projection constants constitute an obstruction to certain normalization procedures for matrix sets. Subsequently, in the recent preprint \cite{epperlein}, Epperlein constructed a three-dimensional real normed space $X$ for which he established, using a computer-assisted argument, the explicit lower bound $1.01$ for the norms of all projections onto two-dimensional subspaces.

It turns out that a remarkable result in this direction predates the question posed by Bosznay and Garay in \cite{bosznaygaray} by a decade. In 1976, Gleit and Lazar \cite{gleitlazar} proved that there exists a universal constant $c>1$ such that every projection from the space $\mathcal{M}_n$ of real $n\times n$ matrices onto a hyperplane has norm at least $c$, when $\mathcal{M}_n$ is equipped either with the operator norm or with the trace-class norm. In particular, they obtained a lower bound independent of the dimension of the matrix space. Their argument combines two key ingredients. The first is an application of von Neumann's mean ergodic theorem, an idea attributed by Gleit and Lazar to Dor \cite{dor}, who used it to give a negative answer to a question of Pełczyński. The second is the rich structure of the matrix space $\mathcal{M}_n$, which contains many $1$-complemented copies of $\mathcal{M}_k$ for $k<n$.

While the result of Gleit and Lazar applies only to the matrix spaces $\mathcal{M}_n$, whose dimensions are necessarily perfect squares, it strongly suggests that a uniform lower bound by a constant independent of dimension should exist for all dimensions greater than $2$. The purpose of the present paper is to show that the underlying mechanism of the Gleit--Lazar argument indeed extends far beyond the matrix spaces in which it was originally discovered. Our main goal is to establish for every $n \geq 3$ the existence of an $n$-dimensional normed space such that the norms of all hyperplane projections are bounded below by $c$, where $c>1$ is independent of the dimension. Moreover, our approach naturally generalizes beyond hyperplanes to subspaces of arbitrary fixed codimension.

Following \cite{epperleinwirth,epperlein}, for $k\geq 1$ and $c>1$ we say that an $n$-dimensional real normed space $X$ is \emph{$c$-shady in codimension $k$} if every projection $P:X\to Y$ with $n-k \leq \dim Y \leq n-1$ satisfies $\|P\|\ge c.$ 
Our main result is the following. 

\begin{twr}
\label{thmproj}
For every integer $k \geq 1$ there exists a constant $c_k>1$ such that for every $n \geq k+2$ there exists an $n$-dimensional real normed space $X$, which is $c_k$-shady in codimension $k$.
\end{twr}

We note that the assumption $n\geq k+2$ is necessary. Indeed, if $n=k+1$, then every subspace of codimension $k$ is one-dimensional and therefore $1$-complemented by the Hahn--Banach theorem.

The proof is inspired by the argument of Gleit and Lazar. The key additional ingredient turns out to be surprisingly simple: finite $\ell_\infty$-sums of shady spaces in codimension $k$ inherit the same property, with the final constant depending only on the indivdual summands. Since every sufficiently large integer can be represented as a linear combination of $k+2$ and $k+3$, and since shady spaces in codimension $k$ exist in these two dimensions, this observation allows us to construct examples in all higher dimensions.

The proof of Theorem~\ref{thmproj} is presented in Section~\ref{sectionmainthm}. We make no attempt to obtain explicit values of the constants $c_k$, although the possibility of deriving quantitative estimates is discussed in Section~\ref{sectionconcluding}.

\section{Proof of Theorem \ref{thmproj}}
\label{sectionmainthm}

We begin with an auxiliary result, which is a straightforward adaptation of Lemma~3.4 from \cite{gleitlazar}. It shows that a uniform lower bound greater than $1$ for the norms of projections automatically yields a similar estimate for a broader class of linear operators. Although the proof is very similar to that of \cite{gleitlazar}, we believe that the underlying idea, based on the mean ergodic theorem and attributed by Gleit and Lazar to Dor \cite{dor}, deserves to be highlighted. The version of the mean ergodic theorem used below may be found in Section VIII.5 of the classical monograph \cite{monograph} by Dunford and Schwartz and is mentioned also in the survey \cite{1complementedsurvey} (see Theorem~5.1 therein). Throughout this section, $I:X\to X$ denotes the identity operator on a given normed space $X$ and for a linear operator $T:X \to X$ by $\fix(T)$ we denote the linear subspace $\ker(T-I)$.

\begin{lem}
\label{lemfix}
Let $k\geq 1$ be an integer and let $X$ be an $n$-dimensional real normed space which is $c$-shady in codimension $k$ for some $c>1$. Then there exists $\delta>0$ such that for any linear operator $T: X \to X$ with $\dim \fix(T) \geq n-k$ and $\|T-I\| \geq 1$, we  have $\|T\| \geq 1 + \delta$.
\end{lem}

\begin{proof} Assume for contradiction that the conclusion is false. Then we can choose a sequence of linear operators $T_m: X \to X$ such that $n-k \leq \dim \fix(T_m) \leq n-1$, $\|T_m - I\| \geq 1$ and $\|T_m\| \leq 1 + \frac{1}{m}$ for every $m \geq 1$. As $\fix(T_m)$ is a non-zero subspace, we also have $\|T_m\| \geq 1$ for all $m$. Since $X$ is finite-dimensional, the set of linear operators $T: X \to X$ satisfying $\|T\| \leq 2$ is compact. Passing to a subsequence, we may therefore assume that $(T_m)$ converges in the operator norm to linear operator $T: X \to X$. As $\|T_m\|\to 1$, we have then $\|T\|=1$, and moreover $\|T-I\| \geq 1$. Since $\dim \fix(T_m) \geq n-k$ we can choose $n-k$ unit vectors $u_1^{(m)}, \ldots u_{n-k}^{(m)} \in \fix(T_m)$, which are pairwise orthogonal with respect to a fixed inner product on $X$. Passing to a further subsequence if necessary, we may assume that $u_j^{(m)} \to u_j$ for $j=1, \ldots, n-k$, where $u_1, \ldots, u_{n-k}$ are still linearly independent. Since $T_m(u_j^{(m)})=u_j^{(m)}$, we obtain
$$T(u_j) = \lim_{m \to \infty} T_m(u_j^{(m)}) = \lim_{m \to \infty} u_j^{(m)} = u_j,$$
so that $u_j \in \fix(T)$ for $j=1, \ldots, n-k$. It follows that $\dim \fix(T) \geq n-k$. Since $\|T\|=1$, we have $\|T^j\|\leq 1$ for every $j\geq 0$. Hence the sequence $(S_N)_{N=1}^{\infty}$ of Ces\`aro means 
$$S_N=\frac{1}{N} \sum_{k=0}^{N-1} T^k$$ 
satisfies $\|S_N\|\leq 1$ for every $N$. As $X$ is finite-dimensional, the mean ergodic theorem yields the convergence of $S_N$ in the operator norm to a linear projection $P:X \to \fix(T)$. Moreover,  $\|P\|\leq 1$ and since every non-zero projection has norm at least $1$, we get $\|P\|=1$. If $\dim\fix(T)<n$, then $P$ is a norm-one projection onto a subspace of codimension at most $k$, contradicting the fact that $X$ is $c$-shady in codimension $k$. Consequently, we must have $\dim \fix(T)=n$ and thus $T=I$. This contradicts the inequality $\|T-I\| \geq 1$ and concludes the proof.
\end{proof}

The crucial property of the matrix space $\mathcal{M}_n$ exploited by Gleit and Lazar is the abundance of $1$-complemented copies of $\mathcal{M}_3$. The following lemma extracts the essential feature of this construction and shows that the argument can be carried out in a much more general setting.

\begin{lem}
\label{lemshady}
Let $k\geq 1$ be an integer and let $X$ be an $n$-dimensional real normed space, where $n \geq k+2$. Let $V_1, \ldots, V_m$ be finite-dimensional real normed spaces such that  every $i=1, \ldots, m$ we have $\dim V_i \geq k+2$ and the space $V_i$ is $c_i$-shady in codimension $k$ for some $c_i>1$. Assume that for every non-zero $x \in X$ there exists a $1$-complemented subspace $V \subseteq X$ containing $x$ and isometric to $V_i$ for some $i\in\{1,\ldots,m\}$. Then $X$ is $c$-shady in codimension $k$ for some constant $c>1$ depending only on the spaces $V_1,\ldots,V_m$.
\end{lem}

\begin{proof}
Let $P:X \to Y$ be a linear projection with $n-k \leq \dim Y \leq n-1$ and let $x \in \ker P$ be non-zero. Choose a subspace $V$ such that $x \in V$ and $V$ is isometric to $V_i$ for some $i \in \{1, \ldots, m\}$. Let $Q: X \to V$ be a projection of norm $1$. Let $U: V_i \to V$ be an isometry and define $T: V_i \to V_i$ as $T = U^{-1}QPU$. Then 
$$\|T\| \leq \|U^{-1}\| \cdot \|Q\| \cdot \|P \| \cdot \|U\| = \|P\|.$$
Moreover, 
$$(T-I)(U^{-1}(x))=U^{-1}QP(x) - U^{-1}(x) = -U^{-1}(x)$$
and hence $\|T-I\| \geq 1$. We now claim that $U^{-1}(Y \cap V) \subseteq \fix(T)$. Indeed, let us take $v \in U^{-1}(Y \cap V)$. Then $U(v) \in Y \cap V$, so that $P(U(v)) = U(v)$ and $Q(U(v))=U(v)$. Therefore
$$T(v) = U^{-1}QPU(v)=U^{-1}Q(U(v))=U^{-1}(U(v))=v.$$
Hence, we conclude that 
$$\dim \fix(T) \geq \dim Y \cap V \geq \dim Y + \dim V - \dim X \geq n-k + \dim V - n = \dim V - k.$$
Lemma~\ref{lemfix} applied to the space $V_i$ now yields
$$ \|T\| \geq 1+\delta_i,$$ 
where $\delta_i>0$ depends only on $V_i$. Consequently, 
$$\|P\|\geq \|T\|\geq 1+\delta_i.$$
It follows that $X$ is $c$-shady in codimension $k$ for $c=1+\min\{ \delta_1, \ldots, \delta_m\}>1$.

\end{proof}

Since the key property responsible for the conclusion of Lemma~\ref{lemshady} has already been isolated, it remains to construct spaces satisfying its assumptions. It turns out that such spaces arise naturally from finite $\ell_\infty$-sums. Recall that if $V_1,\ldots,V_m$ are finite-dimensional normed spaces, then their $\ell_\infty$-sum is the normed space 
$$X=V_1\oplus_{\infty} \ldots \oplus_{\infty} V_m = (V_1\times \ldots \times V_m,\|\cdot\|),$$ 
where
$$\|(x_1,\ldots,x_m)\| = \max\{\|x_i\|_{V_i}: \ i=1,\ldots,m\}$$
for $(x_1,\ldots,x_m)\in X$. The following observation is the main new ingredient of the proof. It shows that every non-zero vector in a finite $\ell_\infty$-sum belongs to a $1$-complemented subspace isometric to one of the summands.

\begin{lem}
Let $V_1, \ldots, V_m$ be finite-dimensional real normed spaces and let
$$X = V_1 \oplus_{\infty} \ldots \oplus_{\infty} V_m.$$
Then for every non-zero $x \in X$ there exists a $1$-complemented subspace $V \subseteq X$ such that $x \in V$ and $V$ is isometric to $V_i$ for some $i \in \{1, \ldots, m\}$.
\end{lem}

\begin{proof}
For a fixed non-zero vector $x=(x_1, \ldots, x_m) \in X$ choose an index $i \in \{1, \ldots, m\}$ such that $\|x\| = \|x_i\|_{V_i}$. Let $f \in V_i^*$ be a linear functional satisfying $\|f\|=1$ and $f(x_i)=\|x_i\|_{V_i}$. For $j \in \{1, \ldots, m\}$, $j \neq i$ we define a mapping $U_j: V_i \to V_j$ as 
$$U_j(v) = \frac{f(v)}{\|x_i\|_{V_i}}x_j \qquad \text { for } v \in V_i.$$ Moreover, let also $U_i:V_i \to V_i$ be the identity on $V_i$. We further define a linear operator $T: V_i \to X$ as
$$T(v) = (U_1(v), \ldots, U_m(v)),$$
for $v \in V_i$. Since $\|x_j\|_{V_j} \leq \|x_i\|_{V_i}$ for any $j$, we have $\|T\| \leq 1$. The opposite inequality follows from the fact that $\|T(v)\| \geq \|U_i(v)\|=\|v\|$ for any $v \in V_i$. Thus, $T$ is an isometry and moreover $T(x_i)=x$. Hence, $V:=T(V_i) \subseteq X$ is a subspace isometric to $V_i$ and containing $x$. It remains to show that $V$ is $1$-complemented in $X$. Clearly, the canonical projection $Q: X \to V_i$ is of norm $1$. Moreover, for $v \in V_i$ we have $QT(v)=U_i(v)=v$. In particular, if we now define $P=TQ: X \to V$, then $P$ is a projection onto $V$, since $P(T(v))=TQT(v)=T(v)$ for $v \in V_i$. Furthermore, we have $\|P\| \leq \|T\| \cdot \|Q\| = 1$, so $\|P\|=1$ and the conclusion follows.
\end{proof}

We are ready to conclude the proof of the main theorem.

\begin{proof}[Proof of Theorem \ref{thmproj}]
For every $i \in \{k+2, k+3\}$, there exists a space $V_i$ which is $c_i$-shady in codimension $k$ for some $c_i > 1$. Since the numbers $k+2$ and $k+3$ are relatively prime, the Frobenius coin theorem implies that every integer $n \geq (k+1)(k+2)$ can be written in the form $a(k+2) + b(k+3)$ for some integers $a, b \geq 0$. For such $n$, let us take 
$$X= \underbrace{V_{k+2}\oplus_\infty\cdots\oplus_\infty V_{k+2}}_{a\text{ times}} \oplus_\infty \underbrace{V_{k+3}\oplus_\infty\cdots\oplus_\infty V_{k+3}}_{b\text{ times}}.$$ 
By the previous lemma and Lemma~\ref{lemshady}, the space $X$ is $c$-shady in codimension $k$ for some constant $c>1$ depending only on $V_{k+2}$ and $V_{k+3}$. On the other hand, for every integer $ k+2\leq n \leq (k+1)(k+2)-1$ there exists an $n$-dimensional real normed space which is $c(n)$-shady in codimension $k$ for some $c(n)>1$. Thus, for every $n \geq k+2$ there exists a $C$-shady space in codimension $k$, where 
$$C = \min \{c, c(k+2), \ldots, c((k+1)(k+2)-1)\}>1.$$
\end{proof}

\section{Concluding remarks}
\label{sectionconcluding}

We conclude the paper with a discussion of the possibility of certain natural extensions of our main result. We have shown that for every integer $k\geq 1$ there exists a constant $c_k>1$ such that for every $n\geq k+2$ there exists an $n$-dimensional normed space which is $c_k$-shady in codimension $k$. However, our argument does not provide explicit values of the constants $c_k$. It seems likely that obtaining explicit estimates would require substantially new ideas.

The hyperplane case appears to be the most accessible for obtaining such estimates. Explicit lower bounds for the shadiness constant in codimension $1$ are already available in fixed dimensions. For $n\geq 4$, such bounds follow from \eqref{estimate}, while for $n=3$ an explicit lower bound of $1.01$ was recently obtained by Epperlein \cite{epperlein}. Nevertheless, these estimates do not lead directly to an explicit version of Theorem~\ref{thmproj}. The main obstacle is Lemma~\ref{lemfix}, whose proof is based on a compactness argument and the mean ergodic theorem. Although the lemma yields the existence of a positive constant $\delta$, it does not provide a quantitative relationship between $\delta$ and the corresponding shadiness constant $c$. On the other hand, since every integer $n \geq 6$ can be written in the form $3a+4b$ for some integers $a, b \geq 0$, the case $k=1$ of the final construction uses only two model spaces $V_3$ and $V_4$. Consequently, an explicit version of Theorem~\ref{thmproj} in the hyperplane case would essentially reduce to establishing a quantitative variant of Lemma~\ref{lemfix}, but only for the particular spaces $V_3$ and $V_4$. From this perspective, the problem may be viewed as extending the known lower bounds for norms of hyperplane projections to a broader class of linear operators with large fixed-point spaces. Although the construction does not cover dimension $5$, this causes no difficulty, since an explicit lower bound is already available in that dimension via \eqref{estimate}. Taking the minimum of the resulting constants would then yield a uniform estimate in every dimension. For higher codimensions, explicit lower bounds greater than $1$ are not known even in fixed dimensions.

Finally, it is natural to ask whether uniform lower bounds can be also obtained for projections onto subspaces of other dimensions. The case of subspaces whose dimensions lie in the range $[\alpha n,\beta n]$, where $\alpha,\beta>0$ are fixed constants, has been extensively studied; see \cite{gluskin,latala,szarek}. In these works, variants of Gluskin's random construction were used to obtain lower bounds of the order $C_{\alpha,\beta}\sqrt n$ for the norms of projections. To the best of our knowledge, however, such methods have not been successfully applied outside this ``middle'' range of dimensions. In particular, we are not aware of any dimension-independent lower bounds for projections onto subspaces of fixed dimension, for instance onto two-dimensional subspaces. It does not seem likely that the method developed in the present paper can be adapted to this setting. In fact, the construction relies crucially on the existence of non-trivial $1$-complemented subspaces, which are used as building blocks in the final $\ell_\infty$-sum argument. Consequently, the method cannot yield a uniform lower bound for projections onto subspaces of \emph{all} dimensions simultaneously. This leads to the following natural question. 
\begin{question} Does there exist a constant $c>1$ such that for every $n\geq 3$ there exists an $n$-dimensional real normed space $X$ satisfying $\|P\|\geq c$ for every projection $P:X\to Y$ with $ 2\leq \dim Y\leq n-1$? 
\end{question} 
A positive answer would almost certainly require ideas that are different from those used in the present paper. On the other hand, it would also be of considerable interest to obtain results in the opposite direction, namely general estimates guaranteeing the existence of a projection of relatively small norm in an arbitrary finite-dimensional normed space. At present, very little seems to be known in this direction.

\section*{Acknowledgements} 

The author would like to thank Microsoft Copilot for its assistance during the preparation of this manuscript, in particular for discussions concerning both mathematical and editorial aspects of the paper. At the same time, it should be emphasized that the essential mathematical ideas underlying the present work were not generated by artificial intelligence. The principal inspiration comes from the paper of Gleit and Lazar \cite{gleitlazar}, as well as from the earlier work of Dor \cite{dor} on the use of the mean ergodic theorem in the study of projections. The author was unaware of the paper of Gleit and Lazar for many years, including during the preparation of the earlier papers \cite{koboslp,kobospolytope} on closely related topics. The paper was discovered only recently and largely by accident. The present work may be viewed as an attempt to isolate and highlight some of the key ideas contained in that remarkable paper, which appears to have received considerably less attention than it deserves.

\end{document}